\documentclass[12pt]{amsart}

\usepackage{amscd,latexsym,amsthm,amsfonts,amssymb,amsmath,amsxtra,color}
\usepackage[mathscr]{eucal}
\usepackage[pdfstartview=FitH]{hyperref}
\usepackage{pdfsync}
\usepackage[backgroundcolor=green!40, linecolor=green!40]{todonotes}
\usepackage{color}
\renewcommand\theequation{\thesection.\arabic{equation}}

\newcommand{\BF}{{\mathbb {F}}}

\makeatletter
\def\Ddots{\mathinner{\mkern1mu\raise\p@
\vbox{\kern7\p@\hbox{.}}\mkern2mu
\raise4\p@\hbox{.}\mkern2mu\raise7\p@\hbox{.}\mkern1mu}}
\makeatother

\newcommand{\CJ}{{\mathcal {J}}}

\newcommand{\CO}{{\mathcal {O}}}
\newcommand{\CP}{{\mathcal {P}}}

\newcommand{\RG}{{\mathrm {G}}}

\newcommand{\Gal}{{\operatorname{Gal}}}
\newcommand{\GL}{{\operatorname{GL}}}

\newcommand{\Ind}{{\operatorname{Ind}}}

\newtheorem{thm}{Theorem}[section]

\newtheorem{lem}[thm]{Lemma}

\newtheorem {ques/conj}[thm]{Question/Conjecture}

\newtheorem{defn}[thm]{Definition}
\newtheorem{rmk}[thm]{Remark}

\makeatletter

\newcommand{\Rmnum}[1]{\expandafter\@slowromancap\romannumeral #1@}
\makeatother

\def\ignore#1{\relax}

\def\bJ{{\boldsymbol{\CJ}}}
\def\bN{{\boldsymbol{N}}}
\def\br{{\boldsymbol{r}}}

\begin{document}
\renewcommand{\theequation}{\arabic{equation}}
\numberwithin{equation}{section}

\title[Local Converse Problem]{On the sharpness of the bound for the Local Converse Theorem of $p$-adic $\GL_N$, general $N$}

\author{Moshe Adrian}
\address{Department of Mathematics\\
Queens College, CUNY\\
Flushing, NY 11367}
\email{moshe.adrian@qc.cuny.edu}

\begin{abstract}
Let $F$ be a non-archimedean local field of characteristic zero.
In this paper we construct examples of supercuspidal representations showing that the bound $[\frac{N}{2}]$ for the local converse theorem of $\GL_N(F)$ is sharp, $N$ general, when the residual characteristic of $F$ is bigger than $N$.
\end{abstract}

\date{\today}
\subjclass[2000]{Primary 11S70, 22E50; Secondary 11F85, 22E55.}
\keywords{Local converse problem, Jacquet's conjecture, Sharpness}
\thanks{Support to Adrian was provided by a grant from the Simons Foundation \#422638 and by a PSC-CUNY award, jointly funded by the
Professional Staff Congress and The City University of New York.}
\maketitle

\section{Introduction}

Let $F$ be a non-archimedean local field of characteristic zero.
Fix a nontrivial additive character $\psi$ of $F$. Given irreducible generic representations $\pi$ and $\tau$ of $\GL_N(F)$ and $\GL_r(F)$, respectively, the twisted-gamma factor $\gamma(s, \pi \times \tau, \psi)$ is defined by using Rankin-Selberg convolution (\cite{JPSS83}) or by using Langlands-Shahidi method (\cite{S84}).
Fix $\pi$, and let $\tau$ be any irreducible generic representation of $\GL_r(F)$, $r \geq 1$. The $\gamma(s, \pi \times \tau, \psi)$ give a set of important invariants of $\pi$. A natural question to ask is how large should $r$ be in order to completely determine $\pi$ using these invariants? This is usually called the {\it Local Converse Problem} for $\RG_N=\GL_N$.  There is much history to this problem (see \cite{ALST18} for a discussion).

Recently (see \cite{JL16} and \cite{Ch16}), the Jacquet conjecture on the local converse problem for $\GL_N$ has been proven, and we have:

\begin{thm}[\cite{JL16}, \cite{Ch16}]\label{lcp1}
Let~$\pi_1,\pi_2$ be irreducible generic representations
of~$\RG_N$. If
\[
\gamma(s, \pi_1 \times \tau, \psi) = \gamma(s, \pi_2 \times \tau, \psi),
\]
as functions of the complex variable~$s$, for all irreducible
generic representations~$\tau$ of~$\RG_r$ with~$r = 1, \ldots,
[\frac{N}{2}]$, then~$\pi_1 \cong \pi_2$.
\end{thm}

By \cite[Section 2.4]{JNS15}, Theorem \ref{lcp1} is shown to be equivalent to
the following theorem with the adjective ``generic" replaced by ``supercuspidal":

\begin{thm}\label{lcp2}
Let~$\pi_1,\pi_2$ be irreducible supercuspidal representations
of~$\RG_N$. If
\[
\gamma(s, \pi_1 \times \tau, \psi) = \gamma(s, \pi_2 \times \tau, \psi),
\]
as functions of the complex variable~$s$, for all irreducible
supercuspidal representations~$\tau$ of~$\RG_r$ with~$r = 1, \ldots,
[\frac{N}{2}]$, then~$\pi_1 \cong \pi_2$.
\end{thm}

In this paper, we show that the bound $[\frac{N}{2}]$ of $r$ is indeed sharp for Theorem \ref{lcp2} when $p > N$, where $p$ is the residual characteristic of $F$. In previous joint work with Liu, Stevens, and Tam (see \cite{ALST18}), we were only able to show this sharpness result in the case that $N$ is prime.  

Precisely, we will construct explicit examples of irreducible supercuspidal representations $\pi_1$ and $\pi_2$ which are not isomorphic, with the property that
 \[
\gamma(s, \pi_1 \times \tau, \psi) = \gamma(s, \pi_2 \times \tau, \psi)
\]
as functions of the complex variable~$s$, for all irreducible
supercuspidal representations~$\tau$ of~$\RG_r$ with~$r = 1, \ldots,
[\frac{N}{2}]-1$.    An analogous result, in the context of $L$ and epsilon factors, has been proven by Henniart \cite{H90}.

As in \cite{ALST18}, we only need to consider the case that $N \geq 5$.  In fact, the examples $\pi_1, \pi_2$ that we construct are precisely the same examples constructed in an earlier unpublished version of \cite{ALST18}.

Let us mention the main difference between this paper and \cite{ALST18}.  In \cite{ALST18}, we assumed that $N$ is prime, and this implied that the relevant Langlands parameters were irreducible.  Therefore, we needed to show that a family of gamma factors of irreducible Langlands parameters were equal.  In this paper, since $N$ is not necessarily prime, the relevant Langlands parameters are not necessarily irreducible.  Therefore, what we need to show is that a family of products of gamma factors of irreducible Langlands parameters are equal.  It goes without saying that this paper was heavily influenced by \cite{ALST18}.

In Section \ref{llp}, we recall some background about supercuspidal representations and Langlands parameters for $GL_N$, and Moy's formula for computing epsilon factors. In Section \ref{construction}, we construct the examples of supercuspidals which will show that the bound $[\frac{N}{2}]$ is sharp for Theorem \ref{lcp2}. The main result (Theorem \ref{main}), that these supercuspidals show that the bound $[\frac{N}{2}]$ is sharp, will be proven in Section
\ref{proofmaintheorem}.

\subsection{Acknowledgements}
We thank Shaun Stevens for discovering a mistake in the paper and helping us to fix it.

\section{Supercuspidal Representations and Local Langlands Parameters}\label{llp}

Let $F$ be a non-archimedean local field of characteristic zero. Let $\CO_F$ be the ring of integers of $F$, $\CP_F$ the maximal ideal in $\CO_F$, and $\BF_q$ the residual field $\CO_F / \CP_F$ with $q = p^f$ being a power of the residual characteristic $p$. In this section, we require that $p \nmid N$.
Let $W_F$ be the Weil group of $F$. 

In the case of $p \nmid N$, there is a nice parametrization of irreducible representations of $W_F$ of dimension $N$ using admissible quasi-characters introduced by Howe (\cite{Ho77}) as follows.

\begin{defn}[Howe \cite{Ho77}]
Let $E/F$ be an extension of degree $N$, $p$ doesn't divide $N$. A quasi-character $\theta$ of $E^{\times}$ is {\rm admissible} with respect to $F$ if
\begin{enumerate}
\item $\theta$ doesn't come via the norm from a proper subfield of $E$ containing $F$,
\item if the restriction $\theta|_{{1+\CP_E}}$ comes via the norm from a subfield $F \subset L \subset E$, then $E/L$ is unramified.
\end{enumerate}
\end{defn}

Two admissible characters $\theta_1$ of $E_1/F$ and $\theta_2$ of $E_2/F$ are said to be conjugate if there is an $F$-isomorphism between $E_1$ and $E_2$ which takes $\theta_1$ to $\theta_2$.

\begin{thm}[Moy, Theorem 2.2.2, \cite{M86}]
Assume that $p$ doesn't divide $N$, $E/F$ an extension of degree $N$. If $\theta$ is an admissible character of $E/F$, then $\rm Ind_{E/F} \theta$ is an irreducible $N$-dimensional representation of $W_F$. Furthermore, two admissible characters induce to equivalent representations if and only if they are conjugate, and each irreducible $N$-dimensional representation $\sigma$ of $W_F$ is induced from an admissible character.
\end{thm}

Fix a non-trivial additive character $\psi_F$ of $F$ of level 1, that is, it is trivial on $\CP_F$ but not on $\CO_F$. For any finite extension $E/F$, define an additive character $\psi_E$ of $E$ via $\psi_E=\psi_F \circ \rm tr_{E/F}$. It is known that if $E/F$ is tamely ramified, then $\psi_E$ is also of level 1, that is, it is trivial on $\CP_E$ but not on $\CO_E$. Let $U_F$ be the group of roots of unity in $F^{\times}$ of order prime to $p$. Fix a uniformizer $\varpi_F$ of $F$ and let $C_F$ be the group generated by $U_F$ and $\varpi_F$. For any finite extension $E/F$, $C_E$ can be defined similarly, which is uniquely determined by $\varpi_F$. The norm map $N_{E/F}$ takes $C_E$ into $C_F$.

Let $E/F$ be a tamely ramified extension of degree $N$, and $\theta$ a quasi-character of $E^{\times}$. Let $c=f_E(\theta)$ be the conductoral exponent of $\theta$, that is, $\theta$ is trivial on $1+\CP_E^c$ but not on $1+\CP_E^{c-1}$. If $c>1$, then there is a unique element $\gamma_{\theta} \in C_E$ such that $\theta(1+x)=\psi_E(\gamma_{\theta}x)$ for $x \in \CP_E^{c-1}$. $\gamma_{\theta}$ is called the standard representative of $\theta$, and we say that $\gamma_{\theta}$ \emph{represents} $\theta$. Note that $\rm val_E(\gamma_{\theta})=1-c$, and if $K/E$ is a finite extension and $\phi = \theta \circ N_{K/E}$, then $\gamma_{\phi} = \gamma_{\theta}$.

\begin{defn}[Kutzko, \cite{Ku80}]\label{generic}
Let $E/F$ be an extension of degree $N$, $p$ doesn't divide $N$. A quasi-character $\theta$ of $E^{\times}$ is {\rm generic} over $F$ if either
\begin{enumerate}
\item $f_E(\theta)=1$, $E/F$ is unramified and $\theta$ doesn't come via the norm from a proper subfield of $E$ containing $F$, or
\item $f_E(\theta)>1$, and $F(\gamma_{\theta}) = E$.
\end{enumerate}
\end{defn}

Note that a generic character is necessarily admissible. If $F \subset L \subset E$ and $\theta$ is a quasi-character of $L^{\times}$, let $\theta_E=\theta \circ N_{E/L}$, or simply $\theta$ when the context is clear. Viewed as a character of $W_L$, $\theta_E$ is the restriction of $\theta$ to the subgroup $W_E$.

\begin{lem}[Howe's Factorization Lemma, \cite{Ho77} and \cite{M86}]
Let $E/F$ be an extension of degree $N$, $p$ doesn't divide $N$.
Let $\theta$ be an admissible character of $E/F$. Then there is a unique tower of fields $F=F_0 \subset F_1 \subset \cdots \subset F_s=E$ and quasi-characters $\chi = \phi_0, \phi_1, \ldots, \phi_s$ of $F^{\times}=F_0^{\times}, F_1^{\times}, \ldots, F_s^{\times}=E^{\times}$ respectively, with $\phi_k$ generic over $F_{k-1}$, $k=1,2,\ldots,s$, such that
$$\theta=\chi \phi_1 \phi_2 \cdots \phi_s.$$
$\phi_k$'s are well-defined modulo a character coming from $F_{k-1}^{\times}$ via the norm map, and the conductoral exponents $f_E(\phi_1) > f_E(\phi_2) > \cdots > f_E(\phi_s)$ are unique.
\end{lem}

In \cite{Ho77}, Howe constructed supercuspidal representations of $\RG_N(F)$ from admissible characters of $E^{\times}$ with respect to $F$ when $E/F$ is tamely ramified of degree $N$. In \cite{M86}, Moy showed that Howe's construction exhausted all the supercuspidal representations in this case. Given an admissible character of $E^{\times}$, denote by $\pi_{\theta}$ the supercuspidal representation constructed by Howe. 

In \cite{M86}, Moy also computed the $\epsilon$-factors of ${\theta}$ as follows.  Let $\theta$ a quasi-character of $F^{\times}$. Let $c=f(\theta)$ be the conductoral exponent of $\theta$, that is, $\theta$ is trivial on $1+\mathfrak{p}_F^c$ but not on $1+\mathfrak{p}_F^{c-1}$.

When $f(\theta) > 1$, let $r = [\frac{f(\theta)+1}{2}]$. Define $c_{\theta} \in \mathfrak{p}_F^{1-f(\theta)}$ $\rm mod \ \mathfrak{p}_F^{1-r}$ such that
$$\theta(1+x)=\psi(c_{\theta}x),$$
for $x \in \mathfrak{p}_F^{r}$. If $E/F$ is a finite extension, then $c_{\theta \circ N_{E/F}} = c_{\theta}$.

Suppose $f(\theta) = f_F(\theta) =  2n+1$ is odd.  Set $H = U_F^n$, $J = U_F^{n+1}$.  We define the Gauss sum $$G(\theta,\psi)=q^{-\frac{1}{2}} \sum_{x \in H/J} \theta^{-1}(x)\psi_F(c_{\theta}(x-1)).$$

 Denote by $\epsilon(\theta, \psi)$ the value of the $\epsilon$-factor at $s=0$. Note that $\epsilon(s, \theta, \psi)=\epsilon(0, \theta \lvert \cdot \rvert^s, \psi)$.

\begin{thm}[Moy, \cite{M86}, (2.3.17)]\label{epsilonfactorformula} $ $
\begin{enumerate}
\item Suppose $f(\theta)$ is even.  Then
\[
\epsilon(\theta, \psi) = \theta^{-1}(c_{\theta}) \psi(c_{\theta}) |c_{\theta}|^{1/2}.
\]
\item Suppose $f(\theta)$ is odd. Then
\[
\epsilon(\theta, \psi) = \theta^{-1}(c_{\theta}) \psi(c_{\theta}) |c_{\theta}|^{1/2} G(\theta,\psi).
\]
\end{enumerate}
\end{thm}

\section{Construction}\label{construction}

In this section, we explicitly construct two non-isomorphic supercuspidal representations of $G_N(F)$ which will later show that the bound $[\frac{N}{2}]$ of $r$ in Conjecture \ref{lcp2} is sharp. From now until the end of the paper, we assume that $p > N$.

Fix a uniformizer $\varpi_F$ in $F$, and define a totally ramified extension $E = F[ \sqrt[N]{\varpi_F} ]$.  Since we assume that $p > N$, in particular, $p \nmid N$, $E$ is a tamely ramified extension of $F$ of degree $N$.
Set $\varpi_E := \sqrt[N]{\varpi_F}$.

Recall that we have assumed that $N \geq 5$.  The element $\beta := \varpi_E^{2-2N}$ defines a character of $1 + \CP_E^{2N-2}$, trivial on $1 + \CP_E^{2N-1}$.  When $N$ is odd, we extend this character in any way to $E^\times$, and denote the resulting extension by $\phi$. That $\phi$ is admissible follows readily since one can see that it is generic (see Definition \ref{generic}). When $N$ is even, the element $\beta_1=\beta=\varpi_E^{2-2N}$ is in $C_{E_1}$, where $E_1=F[\beta_1]\subsetneq E$. We use the same construction as in the odd case to obtain a character $\phi_1$ on $E_1^\times$. We then choose an integer $\ell \in [2,N-1]$ that is coprime to $N$ and define  $\beta_2=\varpi_E^{-\ell}$. This condition implies that $E=E_1[\beta_2]$, and $\beta_2\in C_E$. Similar to above, we define a character $\phi_2$ on $1+\mathcal{P}^{\ell}_{E}$ trivial on $1+\mathcal{P}^{\ell+1}_{E}$ by $\beta_2$, and extend it to a character, still denoted by $\phi_2$, on $E^\times$. Now take $$\phi=(\phi_1\circ N_{E/E_1})\phi_2.$$
 The above constructions imply that $\phi$ is admissible over $F$, and that the product above is the Howe factorization of $\phi$.  We emphasize that we have now defined a character $\phi$, which has a different definition when $N$ is odd from when $N$ is even.

We now define two characters $\phi^{(1)}, \phi^{(2)}$, on $E^{\times}$ by first setting $\phi^{(1)}(\varpi_E) = \phi^{(2)}(\varpi_E)$ and $\phi^{(1)}|_{k_F^{\times}} \equiv  \phi^{(2)}|_{k_F^{\times}}$. We then set $\phi^{(i)}|_{1 + \CP_E^2} \equiv \phi|_{1 + \CP_E^2}$, and then extend $\phi^{(i)}|_{1 + \CP_E^2}$ to $1 + \CP_E$ in two different ways to produce two characters $\phi^{(1)}, \phi^{(2)}$, that differ on $(1 + \CP_E) \setminus (1 + \CP_E^2)$. 

We observe that $\phi^{(1)}$ and $\phi^{(2)}$ can be chosen in a way as to not be isomorphic as admissible pairs. Indeed, there are $q - 1$ nontrivial characters of $1 + \CP_E$ that are trivial on $1 + \CP_E^2$.  However, there are at most $[E:F]$ admissible characters of $E^{\times}$ that are conjugate to any given admissible character of $E^{\times}$ (since two characters of $E^{\times}$ are isomorphic if and only if there exists an $F$-automorphism of $E$ taking one character to the other).  But $[E:F] = N$, and we have assumed from the beginning that $p > N$.  

\begin{rmk}\label{subextensions}
The requirements that we have placed on $\phi^{(1)}, \phi^{(2)}$ imply $\phi^{(1)}|_{F^{\times}} \equiv \phi^{(2)}|_{F^{\times}}$ since $\phi^{(1)}(\varpi_E) = \phi^{(2)}(\varpi_E), \phi^{(1)}|_{k_F^{\times}} \equiv \phi^{(2)}_{k_F^{\times}}$, and $1 + \mathcal{P}_F \subset 1 + \mathcal{P}_E^2$.  We will need this later.
\end{rmk}

Set $\pi_i = \pi_{\phi^{(i)}}$.  We see by the above construction that $\pi_1 \ncong \pi_2$.  
The following is our main result, showing that the bound $[\frac{N}{2}]$ of $r$ in Conjecture \ref{lcp2} is indeed sharp.

\begin{thm}\label{main}
 \begin{equation}\label{gammaequ1}
 \gamma(s, \pi_1 \times \tau, \psi) = \gamma(s, \pi_2 \times \tau, \psi),
 \end{equation}
as functions of the complex variable~$s$, for all irreducible
supercuspidal representations~$\tau$ of~$\RG_r$ with~$r = 1, \ldots,
[\frac{N}{2}]-1$.
\end{thm}

\section{Proof of Theorem \ref{main}}\label{proofmaintheorem}

Henceforth, if $T/F$ is an extension, we denote by $\mathrm{Gal}(T/F)$ the set of all embeddings of $T$ into $\overline{F}$ that fix $F$, whether or not $T/F$ is Galois.  Moreover, for a character $\chi$ of $F^{\times}$, we denote by $\chi_{T/F}$ the character $\chi \circ N_{T/F}$.

\textbf{Proof of Theorem \ref{main}.}
Let $\tau$ be any irreducible
supercuspidal representation of~$\RG_r$, $1 \leq r \leq [\frac{N}{2}]-1$. Since $p \nmid r$, by the discussion in Section 2, one may assume that the local Langlands parameter of $\tau$ is $\Ind_{W_L}^{W_F} \lambda$, for some admissible pair $(L/F, \lambda)$. Let $f_L(\lambda)$ be the conductoral exponent of $\lambda$, and let $m=f_L(\lambda)-1$, so that there exists an element $\alpha \in \CP_L^{-m}$ mod $\CP_L^{-[\frac{m}{2}]}$ that represents $\lambda$, that is,
$$\lambda(1+x) = \psi_L(\alpha x),$$
for $x \in \CP_L^{[\frac{m}{2}]+1}$.
Let $e_L = e(L/F)$ be the ramification index of $L/F$.  Finally, let $M$ be the Galois closure of $K=EL$, over $F$.

To prove \eqref{gammaequ1}, it suffices by the local Langlands correspondence for $GL_N$ to prove that
\begin{equation}\label{gammaequ3}
\gamma(s, \Ind_{W_E}^{W_F} \phi^{(1)} \otimes \Ind_{W_L}^{W_F} \lambda, \psi) = \gamma(s, \Ind_{W_E}^{W_F} \phi^{(2)} \otimes \Ind_{W_L}^{W_F} \lambda, \psi)
\end{equation}
for all such $(L/F, \lambda)$.  

Let $\theta_g^{(i)}=[({}^g \phi^{(i)} \circ N_{L({}^gE)/{}^g E}) \otimes (\lambda \circ N_{L({}^gE)/L})]$, for $g \in W_L\backslash W_F/W_E$, and write $K_g=L({}^gE)$, $E_g = {}^g E$. We note that ${}^g \beta+\alpha$ represents $\theta_g^{(i)}$, for both $i=1,2$.

By Lemma \ref{lem1}, \ref{lem3}, \ref{lem2}, one has to show that
\begin{equation}\label{equ6}
\displaystyle\prod_{g\in W_L\backslash W_F/W_E} {}^g \phi^{(1)} \circ N_{K_g / E_g}({}^g \beta + \alpha) = \displaystyle\prod_{g\in W_L\backslash W_F/W_E} {}^g \phi^{(2)} \circ N_{K_g / E_g}({}^g \beta + \alpha)
\end{equation}

We separate the proof of \eqref{equ6} into three cases: (1) $val_{K}( \beta) = val_{K}(\alpha)$; (2) $val_{K}(\beta) < val_{K}(\alpha)$; (3) $val_{K}(\beta) > val_{K}(\alpha)$, and we note that $val_{K_g}({}^g \beta) = val_K(\beta), val_{K_g}(\alpha) = val_K(\alpha) \ \forall g \in W_L \backslash W_F / W_E$.  Let $e = e(K/E)$ and $N' = e(K/L)$.  We have a diagram of fields, with ramification indices listed on the line segments.

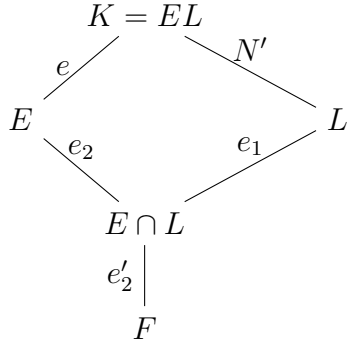
\[
\begin{tikzpicture}[every node/.style={midway}]
  \matrix[column sep={4em,between origins}, row sep={2em}] at (0,0) {
    \node() {}  ; & \node(E L) {$K = E L$}; \\
     \node(E) {$E$}  ; & & \node(L) {$L$}; \\
    \node() {}; & \node(AA) {$E \cap L$};\\
    \node() {}; & \node(F) {$F$};\\
  };
  \draw[-] (E) -- (E L) node[anchor=east]  {$e$};
  \draw[-] (L) -- (E L) node[anchor=south] {$N'$};
  \draw[-] (L) -- (AA) node[anchor=south] {$e_1$};
  \draw[-] (E) -- (AA) node[anchor=south] {$e_2$};
  \draw[-] (F) -- (AA) node[anchor=east] {$e_2'$};
\end{tikzpicture}
\]

We note in particular that since $E/F$ is totally ramified, we have that $N = e_2 e_2'$.  We also have $ \beta \in \CP_{K}^{-e(2N-2)}$ and $\alpha \in \CP_{K}^{-m N'}$.

\textbf{Case (1)}. $val_{K}(\beta) = val_{K}(\alpha)$. Then, $e(2N-2)=mN'$. Since $E/F$ is totally ramified, and by the diamond, we have that $N' = \frac{e e_2}{e_1}$.  Altogether, the equality $e(2N-2)=mN'$ can be rewritten now as 
\[
2e_2 e_2' - 2 = m \frac{e_2}{e_1}.
\]
Multiplying both sides of the equality by $e_1 e_2'$ yields
\[
2N e_1 e_2' - 2 e_1 e_2' = m N.
\]
Simplifying, we get  $N(2e_1 e_2' - m) = 2 e_1 e_2' = 2e_L$. Therefore, we obtain that $N$ divides $2e_L$. On the other hand, we have $1 \leq e_L \leq r < \frac{N-1}{2}$, a contradiction.

\textbf{Case (2)}. $val_{K}(\beta) < val_{K}(\alpha)$. Since ${}^g \phi_{K_g/E_g}^{(1)}({}^g \beta) = {}^g \phi_{K_g/E_g}^{(2)}({}^g \beta)$, we just have to show that

\begin{equation}\label{equ6i}
\displaystyle\prod_{g\in W_L\backslash W_F/W_E} {}^g \phi^{(1)} \circ N_{K_g/E_g}(1 + ({}^g \beta)^{-1} \alpha) = \displaystyle\prod_{g\in W_L\backslash W_F/W_E} {}^g \phi^{(2)} \circ N_{K_g/E_g}(1+ ({}^g \beta)^{-1} \alpha)
\end{equation}

We compute the left hand side.
\begin{align*}
& \ \displaystyle\prod_{g\in W_L\backslash W_F/W_E} {}^g \phi^{(1)} \circ N_{K_g/E_g}(1 + ({}^g \beta)^{-1} \alpha)\\
= & \  \phi^{(1)} \left(\displaystyle\prod_{g\in W_L\backslash W_F/W_E} \left(g^{-1} \left(\prod_{\sigma \in \Gal(K_g/E_g)}(1+\sigma(({}^g \beta^{-1})\alpha))\right)\right)\right)\\
=  & \  \phi^{(1)} \left(\displaystyle\prod_{g\in W_L\backslash W_F/W_E} \prod_{\sigma \in \Gal(K_g/E_g)} g^{-1}(1+{}^g \beta^{-1}\sigma(\alpha))\right)\\
= & \ \phi^{(1)} \left(\displaystyle\prod_{g\in W_L\backslash W_F/W_E} \left(\prod_{\sigma \in \Gal(K_g/E_g)}(1+\beta^{-1} (g^{-1} \sigma)(\alpha))\right)\right)
\end{align*}
since ${}^g \beta^{-1} \in E_g$ is fixed by all $\sigma$.  

We now note that $$W_L\backslash W_F/W_E \cong W_E\backslash W_F/W_L \cong Gal(M/E) \backslash Gal(M/F) / Gal(M/L),$$ where $M$ is the Galois closure of the compositum $EL$ (see \cite[Lemma 2.6]{ALST18}.  We obtain that $W_L\backslash W_F/W_E$ can be interpreted as the orbits in $Gal(M/F) / Gal(M/L) = Gal(L/F)$, under the action of $Gal(M/E)$.  To view this action, we can think of $Gal(L/F)$ as embeddings of $L$ into $M$ (since $M$ is the Galois closure of $EL$) that fix $F$.   

We are therefore considering a set of elements $g^{-1} \sigma$, where $g$ ranges over a set of representatives of the orbits of $Gal(M/E)$ acting on $Gal(L/F)$, and $\sigma \in Gal(K_g/E_g)$.  We wish to show that this set of $g^{-1} \sigma$ exhausts all of $Gal(L/F)$.  In fact it is better to write $g^{-1} \sigma$ as $(g^{-1} \sigma g) g^{-1}$.  We will show that these elements exhaust all of $Gal(L/F)$.

So we now prove that $$\{g^{-1} \sigma : g \in Gal(M/E) \backslash Gal(M/F) / Gal(M/L), \sigma \in \Gal(K_g/E_g) \}$$ can be identified canonically with $Gal(L/F)$, which would simplify the above expression to $$\phi^{(1)} \left(\prod_{\sigma \in \Gal(L/F)}(1+\beta^{-1} \sigma(\alpha))\right).$$
Let $H = Gal(M/E), G = Gal(M/F), J = Gal(M/L)$, so that $H$ acts on $G / J$.  We consider an orbit $Hg^{-1}J, g \in G$.  By the orbit-stabilizer theorem, $Hg^{-1}J \cong \frac{H}{\mathrm{Stab}_H(g^{-1}J)}.$  Furthermore, we argue that $\mathrm{Stab}_H(g^{-1}J) \cong Gal(M / (g^{-1} K_g))$.  Indeed, first note that $x \in Gal(M/E)$ stabilizes $g^{-1} J$ if and only if $x g^{-1} J = g^{-1} J$, i.e. that $x \in g^{-1} J g = Gal(g^{-1} M / g^{-1} L) = Gal(M / g^{-1} L)$, since $g^{-1} M = M$.  Therefore, 

\begin{align*}
& \ x \in \mathrm{Stab}_H(g^{-1}J)\\
\iff & \  x \in Gal(M/E) \cap Gal(M / (g^{-1} L))\\
\iff  & \  x \in Gal(M / ((g^{-1} L) E) = Gal(M / (g^{-1} K_g))
\end{align*}

We conclude that $H g^{-1} J \cong \frac{H}{\mathrm{Stab}_H(g^{-1}J)} \cong \frac{Gal(M/E)}{Gal(M/(g^{-1} K_g))} \cong Gal((g^{-1}K_g)/E)$, noting that $g^{-1} K_g = (g^{-1} L) E$.  Now, via $\eta \mapsto g \eta g^{-1}$, we have $Gal((g^{-1}K_g)/E) \cong Gal(K_g/E_g)$.  The identifications above now allow us to conclude an isomorphism $Gal(K_g / E_g) \xrightarrow{\sim} \frac{H}{\mathrm{Stab}_H(g^{-1}J)}$ given by $\sigma$ maps to the class of $g^{-1} \sigma g$, with $\sigma \in Gal(K_g/E_g)$.  Therefore, the orbit $Hg^{-1} J$ can be identified with the set of $(g^{-1} \sigma g) g^{-1} J$ such that $\sigma \in Gal(K_g / E_g)$, i.e. the set of $g^{-1} \sigma J$, which is exactly what we set out to prove.

So we obtain $$\phi^{(1)} \left(\prod_{\sigma \in \Gal(L/F)}(1+\beta^{-1} \sigma(\alpha))\right).$$

We now have
$$\phi^{(1)} \left( \prod_{\sigma \in \Gal(L/F)}(1+\beta^{-1} \sigma(\alpha)) \right)\\
= \phi^{(1)} \left(1 + \sum_{i=1}^{s}  \beta^{-i} \displaystyle\sum_{H \in B_i}\displaystyle\prod_{\sigma \in H} \sigma(\alpha) \right)$$

where $s = \# Gal(L/F)$, and where $B_i$ is the set of all $i$-element subsets of $Gal(L/F)$.  

Claim: $P(\alpha) :=  \displaystyle\sum_{H \in B_i}\displaystyle\prod_{\sigma \in H} \sigma(\alpha) \in F$.  To see this, firstly, the sum is over $Gal(L/F)$. By construction, $P(\alpha)$ is fixed by every embedding $h$ in $Gal(\overline{L}/F)$ because composition with $h$ just permutes the embeddings in $Gal(L / F)$, so $h$ just permutes the terms in the sum giving $P(\alpha)$.

Since $ \beta^{-1} \in \CP_{E}^{2N-2}$, and $\alpha \in \CP_L^{-m}$, we have
\[
\beta^{-i}  \displaystyle\sum_{H \in B_i}\displaystyle\prod_{\sigma \in H} \sigma(\alpha) \in \CP_{E}^{i(2N-2)+ e_2 e_2' \lceil \frac{-im}{e_1 e_2'} \rceil},
\]
noting that $\CP_{L}^{-im} \cap F = \CP_{F}^{\lceil \frac{-im}{e_1 e_2'} \rceil}$ (this is where we use the above claim that $P(\alpha) \in F$).

We write $$A :=i(2N-2)+ e_2 e_2' \lceil \frac{-im}{e_1 e_2'} \rceil = i(2N-2)+ N \lceil \frac{-im}{e_1 e_2'} \rceil $$  Modulo $N$, we get 
$$  A \equiv -2i \ (mod \ N),$$

Since $1 \leq i \leq s$, we have that, modulo $N$,
$2 \leq -A \leq 2s$.  Note also that $1 \leq r < \frac{N-1}{2}$.  Therefore, we have that, modulo $N$, $2 \leq -A \leq 2s \leq 2r < N-1$.  In particular, one now sees that $A \not\equiv 1 \ \mathrm{mod} \ N$. In particular, $A \neq 1$.  Hence, $A \geq 2$, and $1 + \sum_{i=1}^{s}  \beta^{-i} \displaystyle\sum_{H \in B_i}\displaystyle\prod_{\sigma \in H} \sigma(\alpha) \in 1+\CP_{E}^2$. Therefore, since $\phi^{(1)}, \phi^{(2)}$ agree on $1 + \mathcal{P}_E^2$, we have
$$\phi^{(1)} \left(1 + \sum_{i=1}^{s}  \beta^{-i} \displaystyle\sum_{H \in B_i}\displaystyle\prod_{\sigma \in H} \sigma(\alpha) \right) = \phi^{(2)} \left(1 + \sum_{i=1}^{s}  \beta^{-i} \displaystyle\sum_{H \in B_i}\displaystyle\prod_{\sigma \in H} \sigma(\alpha) \right),$$
completing the proof of Case (2).

%

\textbf{Case (3)}. $val_{K}(\beta) > val_{K}(\alpha)$.  This proof here is similar to the proof of case (2).  We note that 
\begin{align*}
& \ \displaystyle\prod_{g\in W_L\backslash W_F/W_E} {}^g \phi^{(i)} \circ N_{K_g/E_g}(\alpha)\\
 = & \ \phi^{(i)} \left(\displaystyle\prod_{g\in \mathrm{Gal}(M/E) \backslash \mathrm{Gal}(L/F)} \prod_{\sigma \in \Gal(K/E)} \sigma(\alpha))\right)\\
  = & \ \phi^{(i)} \left(\displaystyle\prod_{g\in \mathrm{Gal}(K/E) \backslash \mathrm{Gal}(L/F)} \prod_{\sigma \in \Gal(K/E)} \sigma(\alpha))\right)\\
    = & \ \phi^{(i)} \left(\prod_{\sigma \in \Gal(L/F)} \sigma(\alpha)\right)\\
=  & \ \phi^{(i)}(N_{L/F}(\alpha)),
\end{align*}
and since $\phi^{(1)}|_{F^{\times}} \equiv \phi^{(2)}|_{F^{\times}}$ (see Remark \ref{subextensions}), we see that \eqref{equ6} is equivalent to 
$$\displaystyle\prod_{g\in W_L\backslash W_F/W_E} {}^g \phi^{(1)} \circ N_{K_g / E_g}(1 + {}^g \beta \alpha^{-1}) = \displaystyle\prod_{g\in W_L\backslash W_F/W_E} {}^g \phi^{(2)} \circ N_{K_g / E_g}(1 + {}^g \beta \alpha^{-1})$$

But the argument now is the same is in case (2).  On the left, we arrive at the term 
$$\phi^{(1)} \left( \prod_{\sigma \in \Gal(L/F)}(1+\beta \sigma(\alpha^{-1})) \right),$$
and also the term (using the same notation as in case (2)) $A = -( i(2N-2)+ e_2 e_2' \lceil \frac{-im}{e_1 e_2'} \rceil)$, which is precisely the negative of the term $A$ that appeared in case (2).  We still have $A > 0$, since $val_K(\beta) > val_K(\alpha)$.  Then, $A \equiv 2i \ (mod \ N)$, and $2 \leq A < N-1$, so that $A \not\equiv 1 \ \mathrm{mod} \ N$ and thus $A \neq 1$.  Hence, we have proven Case (3).

This completes the proof of Theorem \ref{main}, up to the the proofs of Lemmas \ref{lem1}, \ref{lem3}, and \ref{lem2}.
\qed

\subsection{Lemmas}
In this section, we lay out the various Lemmas that we needed in the proof of Theorem \ref{main}.

\begin{lem}\label{lem1}\cite[Lemma 2.5]{ALST18}
For $i=1,2$, $(Ind_{W_E}^{W_F} \phi^{(i)} ) \otimes (Ind_{W_L}^{W_F} \lambda)$ is isomorphic  to

\[
\bigoplus_{g\in W_L\backslash W_F/W_E} Ind_{W_{L({}^gE)}}^{W_F} [({}^g \phi^{(i)} \circ N_{L({}^gE)/{}^g E}) \otimes (\lambda \circ N_{L({}^gE)/L})].
\]
\end{lem}

Recall that $\theta_g^{(i)}=[({}^g \phi^{(i)} \circ N_{L({}^gE)/{}^g E}) \otimes (\lambda \circ N_{L({}^gE)/L})]$, for $g \in W_L\backslash W_F/W_E$, and $K_g = L({}^gE)$, $E_g = {}^g E$.  We also write $\nu_g^{(i)} = [({}^g \phi^{(i)} \circ N_{M_g/E_g}) \otimes (\lambda \circ N_{M_g/L})]$ for $i = 1, 2$.

Since the $L$-function is inductive, multiplicative, and trivial on ramified characters, and by \cite[Lemma 2.4]{ALST18}, we have

\begin{lem}\label{lem3}
We have for $i = 1,2$,
\[
\gamma(s, \bigoplus_{g\in W_L\backslash W_F/W_E} \Ind_{W_{K_g}}^{W_F} \theta_g^{(i)}, \psi_F) = \displaystyle\prod_{g\in W_L\backslash W_F/W_E}
\lambda_{K_g/F}(\psi_F)\gamma(s,\theta_g^{(i)}, \psi_{K_g}),
\]
where $\lambda_{K_g/F}(\psi_F)$ is the Langlands constant associated to $K_g/F$ and $ \psi_F$ (see \cite[Theorem 29.4]{BH06}, \cite[\S34.3]{BH06}), and $\psi_{K_g}=\psi_F \circ \mathrm{tr}_{K_g/F}$. The formula is also true when $\gamma$ is replaced by $\epsilon$.
\end{lem}

\begin{lem}\label{lem2}
\[
\frac{\gamma(s, \theta_g^{(1)}, \psi_{K_g})}{\gamma(s, \theta_g^{(2)}, \psi_{K_g})} = \frac{\theta_g^{(1)}({}^g \beta + \alpha)}{ \theta_g^{(2)}({}^g \beta + \alpha)}
\]
\end{lem}

\begin{proof}
It suffices to prove the formula with $\gamma$ replaced by $\epsilon$, by the proof of Lemma \ref{lem3}.  

We first note that $c_{\nu_g^{(1)}} = c_{\nu_g^{(2)}} = {}^g\beta + \alpha$ so that in particular $f := f_{K_g}(\nu_g^{(1)}) = f_{K_g}(\nu_g^{(2)})$. If $f$ is even, then the the proposed equality is easily seen to be true by Theorem \ref{epsilonfactorformula}.

Suppose that $f=2n+1$ is odd. We must check that $G(\theta_g^{(1)}, \psi_{K_g}) = G(\theta_g^{(2)}, \psi_{K_g})$. The impact of $g \in W_L \backslash W_F / W_E$ is negligible, so we will assume that $g = 1$.  

Set $\theta^{(i)} =  \phi^{(i)} \circ N_{LE/ E} \otimes \lambda \circ N_{LE/L},$ $K = EL$, and $f = f_K(\theta^{(1)}) = f_K(\theta^{(2)}) = 2n+1$ is odd. 
Since $c_{\theta^{(1)}} = c_{\theta^{(2)}} = \beta + \alpha$, and by definition of $\theta^{(i)}$, it suffices to check that 
\begin{equation}\label{formula1}
(\phi^{(1)} \circ N_{K/E})(x) = (\phi^{(2)} \circ N_{K/E})(x) \ \mathrm{for} \ x \in U_{K}^n = 1+\CP_{K}^n.
\end{equation}
We first note that $f = \mathrm{max} \{e(2N-2), mN' \} +1$.  Therefore, since $N \geq 5$, we have that $n \geq \frac{3}{2} > e$, so that $N_{K/E}(U_K^n) \subset U_E^2$ by a standard property of the norm map.  Since $\phi^{(1)}$ and $\phi^{(2)}$ agree on $U_E^2$, we have that
\eqref{formula1} is proven.
\end{proof}

\end{document}